\documentclass{amsart}
\usepackage{amsmath,amssymb,epsfig,bbm}

\theoremstyle{plain}
\newtheorem{theorem}{Theorem}
\newtheorem{corollary}[theorem]{Corollary}
\newtheorem{lemma}[theorem]{Lemma}
\newtheorem{proposition}[theorem]{Proposition}

\theoremstyle{definition}
\newtheorem{definition}[theorem]{Definition}

\theoremstyle{remark}
\newtheorem{remark}[theorem]{Remark}

\def\qed{\hfill $\Box$}

\newcommand{\dueto}[1]{\textup{\textbf{(#1) }}}
\newcommand{\nin}{\not\in}

\newcommand{\tmabbr}[1]{#1}
\newcommand{\tmop}[1]{\ensuremath{\operatorname{#1}}}
\newcommand{\tmtextit}[1]{{\itshape{#1}}}

\begin{document}

\title[Local recognition of reflection graphs]{The local recognition of reflection graphs of spherical Coxeter
groups}

\author{Ralf Gramlich}
\thanks{The first author gratefully acknowledges a
Heisenberg Fellowship by the Deutsche Forschungsgemeinschaft.}
\address{Technische Universit\"at Darmstadt, FB Mathematik, Schlo{\ss}gartenstra{\ss}e 7, 64289 Darmstadt, Germany + University of Birmingham, School of Mathematics, Edgbaston, Birmingham, B15 2TT, United Kingdom}
\email{gramlich@mathematik.tu-darmstadt.de + ralfg@maths.bham.ac.uk}

\author{Jonathan I. Hall}
\thanks{The second author gratefully acknowledges 
partial support provided by the National Science Foundation (USA)}
\address{Michigan State University, Department of Mathematics, East Lansing, Michigan 48824, U.S.A.}
\email{jhall@math.msu.edu}

\author{Armin Straub}
\address{Tulane University, Department of Mathematics, 6823 St.\ Charles Ave, New Orleans, LA 70118, U.S.A.}
\email{math@arminstraub.com}

% % http://www.ams.org/msc/
% % 05C25 Graphs and groups
% % 05C75 Structural characterization of types of graphs
% % 20F55 Reflection and Coxeter groups
% % 20E25 Local properties
\subjclass[2000]{05C25, 05C75, 20F55, 20E25}

\keywords{Local recognition of graphs, Coxeter groups}

\begin{abstract}
Based on the third author's thesis \cite{straub-diploma} in this article we complete the local recognition of commuting reflection graphs of spherical Coxeter groups arising from irreducible
crystallographic root systems.
\end{abstract}

\maketitle

\section{Introduction}

Given a connected graph one may ask to which extent it is determined by its
local graphs, that is, by the induced subgraphs on the vertices adjacent to a
particular vertex. This local recognition of graphs has been studied
extensively in the literature, for instance in {\cite{buekenhout-gt77}}, \cite{Cuypers/Pasini:1992},
{\cite{hall-gt85b}, \cite{Pasechnik:1994}, \cite{Weetman:1994a}, \cite{Weetman:1994b}} to mention a few; see also \cite{brown-gt75}, \cite{Cohen:1990}, \cite{hall-gt85}. A particularly guiding example for the topic of the
present article is the local recognition of the Kneser graphs studied in
{\cite{hall-gt80}} and {\cite{hall-gt87}}.

We are interested in the local recognition of Weyl graphs, i.e., graphs on the reflections of Coxeter groups with the commutation relation as adjacency. 
A combination of our findings with results from \cite{buekenhout-gt77}, \cite{hall-gt87}, \cite{hall-gt85b} yields the following recognition result.

\medskip
\noindent 
{\bf Main Theorem.} \\
{\em  The following are true up to isomorphism.
  \begin{itemize}
    \item A Weyl graph of type $A_n$ ($n \geqslant 8$), type $B_n, C_n$
    ($n = 3$ or $n \geqslant 5$), type $D_n$ ($n \geqslant 9$), or type
    $E_7$ is uniquely determined, as a connected graph, by its local graphs.

    \item A Weyl graph of type $A_6$, $A_7$, $D_7$, $D_8$, $E_6$, $E_8$ is
    uniquely determined by its local graphs and its size.

    \item The Weyl graph $W$ of type $F_4$ and its twisted copy (defined at the
    end of section \ref{4.1}) are the only bichromatic graphs of
    size $24$ with local graphs like $W$.
  \end{itemize}
}

\medskip
The remaining small Weyl graphs of type $A_n$ as well as those of types $I_2(m)$, $G_2$,
$H_3$, $H_4$ are locally a disjoint union of complete graphs. The graphs of type $D_n$
are obtained as doubles of those of type $A_{n - 1}$, so
that local recognition results for type $A_n$ transfer to $D_n$. Finally,
types $B_4$ and $C_4$ are treated in Remark \ref{rk-locallylikewb4}.

The local recognition of the Weyl graphs of type $A_7$, $E_6$ and $E_8$
has been established in the fundamental work
{\cite{buekenhout-gt77}}. The case of $A_6$, for which the Weyl graph is locally the Petersen graph, has been
studied in {\cite{hall-gt80}}. Weyl graphs of types $A_n$ and $E_n$ which
are locally cotriangular have been treated in {\cite{hall-gt85b}}. The local
recognition of types $B_n$ and $C_n$ is proved in Theorem \ref{thm-lr-recognizewbn}.
The Weyl graph of type $F_4$ is not uniquely determined by its local graphs (Corollary \ref{thm-infinitelymanygraphslocallylikewf4}). We nevertheless characterize this Weyl graph as as one of two tightest graphs with the prescribed local structure (Theorem
\ref{thm-recognizewf4}).
In the last section we turn to group theoretical applications of local
recognition results for Weyl graphs. 

\subsection*{Acknowledgement} The authors thank the referee for several suggestions that helped to improve the exposition of this article.

\section{\label{sct-localrecognition}Local recognition of graphs}

All graphs considered in this text are simple and undirected. We use $\perp$
to denote adjacency, and our notation for operations on graphs like the
cartesian product or joins follows {\cite{harary-gt94}}. Let $\Gamma$ be a
graph, and $x \in \Gamma$ a vertex. We write $x^{\perp}$ to denote the set
of neighbors of $x$, that is, the set of vertices adjacent to $x$. Likewise,
for $X \subseteq \Gamma$ we write $X^{\perp} = \bigcap_{x \in X} x^{\perp}$.
The induced subgraph on $x^{\perp}$ is
called the {\em local graph} at $x$. A graph $\Gamma$ is
said to be {\em locally homogeneous}, if there exists a graph $\Delta$ such that each local graph of $\Gamma$ is isomorphic to $\Delta$. In this case, $\Gamma$
is said to be locally $\Delta$, and $\Delta$ is referred to as the local graph
of $\Gamma$. If $\Gamma$ is locally homogeneous, then we denote its local graph
by $\Delta (\Gamma)$.

In this article we are interested in the problem of characterizing a connected locally
homogeneous graph in terms of its local graph. We say that a connected locally homogeneous graph $\Gamma$ is
{\em locally recognizable}, if up to
isomorphism $\Gamma$ is the only connected graph that is locally $\Delta
(\Gamma)$. In case $\Lambda$ is another locally homogeneous graph such that
$\Delta (\Lambda) \cong \Delta (\Gamma)$ we say that $\Lambda$ is
{\em locally like $\Gamma$}.

The above terminology naturally extends to bichromatic graphs. For reasons that
become clear later, we distinguish the vertices of a bichromatic graph as
{\em short} versus {\em long}. All morphisms between bichromatic graphs
are understood to preserve this distinction. We say that a bichromatic graph is
{\em locally homogeneous}, if the local graphs at short vertices are all
isomorphic to some bichromatic graph $\Delta_s$ and the local graphs at
long vertices are all isomorphic to some bichromatic graph $\Delta_{\ell}$. In
this case we say that $\Delta_s$ is the {\em short local graph}
of $\Gamma$ and that $\Delta_{\ell}$ is the {\em long local
graph} of $\Gamma$. If $\Gamma$ is a bichromatic
locally homogeneous graph, then we denote its short local graph by $\Delta_s
(\Gamma)$ and its long local graph by $\Delta_{\ell} (\Gamma)$. If $\Lambda$ is
another bichromatic locally homogeneous graph such that the short as well as
the long local graphs of $\Lambda$ and $\Gamma$ are isomorphic as bichromatic graphs, then we say
that $\Lambda$ is {\em locally like $\Gamma$}. Finally, given a graph $\Gamma$
we denote with $\Gamma^s$ and $\Gamma^\ell$ the bichromatic graphs obtained from
$\Gamma$ with all vertices treated as short respectively long.

One easily verifies that the Kneser graph $K (n, k)$ is locally homogeneous
with local graph $K (n - k, k)$. The second author proved in
{\cite{hall-gt87}} that for $n$ sufficiently large compared to $k$ the Kneser
graphs are locally recognizable; for $k = 2$, it sufficies to require $n \geq 7$. In {\cite{hall-gt80}} he classified the
three connected graphs which are locally the Petersen graph $K (5, 2)$.
The classification of graphs that are locally $K (6, 2)$ is contained
in {\cite{buekenhout-gt77}}.

\begin{theorem}
  {\dueto{{\cite{hall-gt87}}, {\cite{hall-gt80}}, \cite{buekenhout-gt77}}}
  \label{thm-lr-recognizeknesergraphs}\label{thm-lr-locallyk52}\label{thm-lr-locallyk62}Let
  $k \geqslant 1$, and $\Gamma$ be a connected graph that is locally $K (n,
  k)$.
  \begin{itemize}
    \item If $n \geqslant 3 k + 1$ then $\Gamma \cong K (n + k,
    k)$.
    
    \item If $(n, k) = (5, 2)$ then $\Gamma$ is isomorphic to one of the
    graphs $K (7, 2)$, $3 \cdot K (7, 2)$, or $\Sigma L_{2, 25}$. In
    particular, $| \Gamma | \in \{21, 63, 65\}$.
    
    \item If $(n, k) = (6, 2)$ then $\Gamma$ is isomorphic to one of the
    graphs $K (8, 2)$, $\mathcal{S}p_6 (2)$ minus $\{x\} \cup x^{\perp}$ for
    some $x$, or $\mathcal{N}^-_6 (2)$. In particular, $| \Gamma |
    \in \{28, 32, 36\}$. \qed
  \end{itemize}
\end{theorem}

Here, the graph $3 \cdot K (7, 2)$ is the $3$-fold cover of $K (7, 2)$, and
$\Sigma L_{2, 25}$ is the graph on the conjugates of the unique non-trivial
field automorphism of $\mathbbm{F}_{25}$ in the special semilinear
group $\Sigma L (2, 25)$ with two elements
adjacent whenever they commute. More details can be found in
{\cite{hall-gt80}}. Further, the graph $\mathcal{S}p_{2 n} (2)$ is the graph on
the non-zero vectors of $V = \mathbbm{F}_2^{2 n}$ with two vectors adjacent
whenever they are perpendicular with respect to a non-degenerate symplectic
form $B$ on $V$. Up to isomorphism there are only two quadratic forms $Q^+$
and $Q^-$, corresponding to maximal or minimal Witt index, on $V$ that $B$ is
associated to, and the graph $\mathcal{N}^{\varepsilon}_{2 n} (2)$
is the induced subgraph of $\mathcal{S}p_{2 n} (2)$ on the
vectors that are non-singular under $Q^{\varepsilon}$. For more details about
these graphs we refer to {\cite{hall-gt85b}}.

Ernest E{\tmabbr{.}} Shult and the second author actually proved a lot more in
{\cite{hall-gt85b}}. They characterize the graphs that are locally
cotriangular in the following sense. A graph is said to be
{\em cotriangular}, if every pair $x, y$ of
non-adjacent vertices is contained in a cotriangle, that is, a $3$-coclique
$\{x, y, z\}$ such that every other vertex is adjacent to either all or
exactly one of the vertices $x, y, z$. Observe that a join
$\Gamma + \Lambda$ is cotriangular if and only if both $\Gamma$ and $\Lambda$
are. Denote with $\Gamma^{\ast}$ the {\em reduced graph}
of $\Gamma$, that is, the graph on the equivalence classes of vertices
of $\Gamma$ with the same closed neighborhood and two classes adjacent
whenever some representatives are adjacent. Then $\Gamma$ is cotriangular if
and only if $\Gamma^{\ast}$ is. A graph $\Gamma$ is called completely reduced
in this context whenever $\Gamma^{\ast} = \Gamma$ and $\Gamma$ can not be
decomposed into $\Gamma_1 + \Gamma_2$ with non-empty $\Gamma_1, \Gamma_2$. A
classification of all cotriangular graphs is given by the following theorem
due to Ernest E{\tmabbr{.}} Shult.

\begin{theorem}
  {\dueto{{\cite{shult-gt74}}}}A finite completely reduced
  graph is cotriangular if and only if it is isomorphic to one of the graphs
  \[ K (n, 2), \hspace{1em} n \geqslant 2 ; \hspace{1em} \mathcal{S}p_{2 n} (2),
     \hspace{1em} n \geqslant 2 ; \hspace{1em}
     \mathcal{N}^{\varepsilon}_{2 n} (2), \hspace{1em} \varepsilon =
     \pm 1, n \geqslant 3.  \] \qed
\end{theorem}

The graphs $K (2, 2) \cong K_1$ and $K (3, 2) \cong \overline{K_3}$ are
considered degenerate. Let $\mathcal{D}$ denote the set of graphs $\Gamma$
such that $\Gamma^{\ast}$ is a finite completely reduced cotriangular graph.
If $\mathcal{G}$ is a collection of graphs, then we say that a graph $\Gamma$
is locally $\mathcal{G}$ if for each $x \in \Gamma$ the local graph at $x$ is
isomorphic to some graph of $\mathcal{G}$.

\begin{theorem}
  {\dueto{{\cite[Main
  Theorem]{hall-gt85b}}}}\label{thm-lr-locallycotriangular}Let $\Gamma$ be
  connected and locally $\mathcal{D}$. Then either $\Gamma$ is locally $\{K_1,
  \overline{K_3} \}$ or $\Gamma$ is isomorphic to one of the following graphs
  \begin{itemize}
    \item $K (n, 2)$ where $n \geqslant 7$,
    
    \item $\mathcal{S}p_{2 n} (2)$ possibly with a polar subspace deleted,
    
    \item $\mathcal{H}_{2 n}^{\varepsilon} (T)$, $\mathcal{G}_{2
    n}^{\varepsilon}$,
    
    \item $3 \cdot K (7, 2)$, $\Sigma L_{2, 25}$, or $\mathcal{N}_6^+ (3)$.
  \end{itemize}
\qed\end{theorem}

The graphs $\mathcal{H}_{2 n}^{\varepsilon} (T)$, $\mathcal{G}_{2
n}^{\varepsilon}$ are derived from the graph $\mathcal{S}p_{2 n} (2)$; see {\cite{hall-gt85b}}. Note that the case $k = 2$ of Theorem
\ref{thm-lr-recognizeknesergraphs} can be regarded as a special case of the
classification in Theorem \ref{thm-lr-locallycotriangular}. The
following special case of Theorem \ref{thm-lr-locallycotriangular} has
already been established in \cite{buekenhout-gt77} by Francis Buekenhout and 
Xavier Hubaut. 

\begin{theorem}
  {\dueto{{\cite[Theorem 2 (3)]{buekenhout-gt77}}}}\label{thm-lr-locallysymplectic}Let
  $\Gamma$ be connected and locally $\mathcal{S}p_{2 n} (2)$ for some $n
  \geqslant 2$. Then $\Gamma$ is isomorphic to one of the following graphs
  $\mathcal{N}^+_{2 n + 2} (2)$, $\mathcal{N}^-_{2 n + 2} (2)$, or
  $\mathcal{S}p_{2 n + 2} (2)$ minus $\{x\} \cup x^{\perp}$ for some
  $x$.
\qed\end{theorem}

The preceding theorem has been generalized in \cite{Cohen/Shult:1990}, \cite{Cuypers/Pasini:1992}.

\section{\label{sct-recognizeweylgraphs}Local recognition of Weyl graphs}

We assume that the reader is familiar with Coxeter groups and root systems as treated in {\cite{humphreys-gr92}} or
{\cite{bourbaki-lie4}}. The {\em commuting graph}
of a group $G$ on $X \subseteq G$ is the graph with vertex set $X$ in which two
vertices $g, h \in X$ are adjacent whenever $g$ and $h$ commute. We will study
the commuting graphs of finite Coxeter groups on their reflections. Since we
are interested in local recognition results we will focus on finite
irreducible Coxeter groups for which the reflection graph is locally
homogeneous. The graphs arising from the cases $H_3$, $H_4$ and $I_2 (m)$ are
locally disjoint unions of complete graphs and therefore not interesting for the
purpose of local recognition.
Hence, we further restrict to Coxeter groups which arise from irreducible
crystallographic root systems. These are those with Dynkin diagram equal to
one of $A_n$ ($n \geqslant 1$), $B_n$ or $C_n$ ($n \geqslant 2$), $D_n$
($n \geqslant 4$), $E_6$, $E_7$, $E_8$, $F_4$, or $G_2$.

Recall that each root of an irreducible crystallographic root system $\Phi$
is considered either short or long (with the convention
that in the absence of two distinct root lengths every root is long). If $M$ is
the Dynkin diagram of $\Phi$ then we denote with $W(M)$ the Weyl group of $\Phi$,
{\tmabbr{i.e.},} the group generated by the reflections through the roots of
$\Phi$, together with the notion of a short (respectively long) root reflection
by $W(M)$. The {\em Weyl graph} $\mathbbm{W}(M)$ is the commuting graph of
$W(M)$ on its reflections. If $M$ is simply laced then all reflections in
$W(M)$ are conjugate, which implies that the Weyl graph $\mathbbm{W}(M)$ is
locally homogeneous. On the other hand, if $M$ is not simply laced then there
are two conjugacy classes of reflections in $W(M)$, namely short and long root
reflections, and we regard $\mathbbm{W}(M)$ as a bichromatic graph. Instead of
assigning arbitrary colors we accordingly refer to the vertices of $\mathbbm{W}(M)$
corresponding to short (respectively long) root reflections as short (respectively
long) vertices. As a bichromatic graph, the Weyl graph $\mathbbm{W}(M)$ is
locally homogeneous.

$\mathbbm{W}(A_n)$ is the graph with vertices $y_{i, j}$, $1 \leqslant i < j
\leqslant n + 1$, such that $y_{i, j} \perp y_{k, l}$ if and only if $\{i,
j\} \cap \{k, l\}= \emptyset$. Consequently, the Weyl graph
$\mathbbm{W}(A_n)$ is isomorphic to the Kneser graph $K (n + 1, 2)$. Likewise,
$\mathbbm{W}(D_n)$ is the graph with vertices $y_{i, j}$, $1 \leqslant i,
j \leqslant n$, such that $y_{i, j} \perp y_{k, l}$ if and only if $\{i, j\}
\cap \{k, l\}= \emptyset$ or $(k, l) = (j, i)$. $\mathbbm{W}(D_n)$ is therefore
isomorphic to the composition graph $K (n, 2) [K_2]$, that is,
the graph arising from the Kneser graph $K (n, 2)$ by replacing each vertex
by an adjacent pair of vertices. Accordingly, Theorem \ref{thm-lr-recognizeknesergraphs}
applies and yields the recognition results of the Main Theorem
for types $A_n$ and $D_n$.
By \cite{buekenhout-gt77} we have
$\mathbbm{W}(E_6) \cong \mathcal{N}^-_6 (2)$, $\mathbbm{W}(E_7) \cong \mathcal{S}p_6 (2)$
and $\mathbbm{W}(E_8) \cong \mathcal{N}^+_8 (2)$. The corresponding recognition
results of the Main Theorem follow from Theorems \ref{thm-lr-locallyk62},
\ref{thm-lr-locallycotriangular} and \ref{thm-lr-locallysymplectic}.

$\mathbbm{W}(B_n)$ is the bichromatic graph with vertices $y_{i,
j}$, $1 \leqslant i, j \leqslant n$, where the $y_{i, i}$ are short and the
$y_{i, j}$ with $i \neq j$ are long vertices, and $y_{i, j} \perp y_{k, l}$
if and only if $\{i, j\} \cap \{k, l\}= \emptyset$ or $(k, l) = (j, i)$. The
Weyl graph $\mathbbm{W}(C_n)$ is obtained from $\mathbbm{W}(B_n)$ by
exchanging the role of short and long vertices. The recognition results of
the Main Theorem for types $B_n$ and $C_n$ are therefore contained
in the following theorem.

\begin{theorem}
  \label{thm-lr-recognizewbn}Let $n = 3$ or $n \geqslant 5$, and let $\Gamma$
  be a connected bichromatic graph which is locally like $\mathbbm{W}(B_n)$.
  Then $\Gamma \cong \mathbbm{W}(B_n)$.
\end{theorem}

\begin{proof}
  It is straightforward to check the case $n = 3$.
  
  Next, let $n \geqslant 6$. Let $X$ be a short component of $\Gamma$ and $x
  \in X$ a short vertex. The short induced subgraph on $x^{\perp}$ is a clique
  on $n - 1$ elements which implies that $X$ is a clique on $n$ elements.
  By assumption, the long neighbors of $x$ induce a subgraph isomorphic to the
  long induced subgraph of $\mathbbm{W}(B_{n - 1})$. This subgraph is
  isomorphic to $\mathbbm{W}(D_{n - 1})$ and, in particular, is connected for
  $n \geqslant 6$. This implies that all
  long neighbors of $x$ are contained in a single long component $Y$ of
  $\Gamma$. Consider a short vertex $x_1 \in X$ adjacent to $x$. Again, all
  long neighbors of $x_1$ lie in one long component of $\Gamma$. But looking
  at $\{x, x_1 \}^{\perp} \subset x^{\perp}$ we see that $x$ and $x_1$ share
  long neighbors whence this component has to be $Y$ as well. Since $X$ is
  connected this shows that all long vertices adjacent to some vertex of $X$
  are contained in $Y$. Likewise, let $y \in Y$. The short induced subgraph of
  $y^{\perp}$ is a clique on $n$ vertices and thus in particular connected.
  Again, we see that for a long vertex $y_1$ adjacent to $y$ the common
  neighbors $\{y, y_1 \}^{\perp}$ contain a short vertex. Therefore the same
  argument as before shows that all short vertices adjacent to some vertex of
  $Y$ are contained in $X$. Since $\Gamma$ is connected this proves that $X$
  and $Y$ are the only short respectively long components of $\Gamma$.
  
  We count the number of long vertices by counting the long neighbors of the
  $n$ short vertices of $\Gamma$. By assumption, a short vertex has $(n -
  1) (n-2)$ long neighbors. Further, two short vertices have $(n-2) (n-3)$ long
  neighbors in common, three short vertices have $(n-3) (n-4)$ long
  neighbors in common, and so on. Thus there are
  \[ \binom{n}{1} (n-1)(n-2) - \binom{n}{2} (n-2)(n-3) + \ldots + (-
     1)^{n-1} \binom{n}{n-2} 2 = n (n-1) \]
  long vertices in $\Gamma$. Note that for the above equation we exploited
  that the alternating sum of the binomial coefficients equals zero, that is,
  $\sum_{k = 0}^n (- 1)^k  \binom{n}{k} = 0$.
  
  Let $x_1, x_2, \ldots, x_n$ be the short vertices of $\Gamma$.
  $\Gamma$ is locally $\mathbbm{W}(B_{n - 1})$ at short vertices which implies
  that for $1 \leqslant i \neq j \leqslant n$ the common neighborhood
  $\{x_r : r \nin \{i, j\}\}^{\perp}$ contains exactly two long vertices which
  we denote by $y_{i, j}$ and $y_{j, i}$. Since a long vertex is adjacent to
  exactly $n-2$ short vertices the $y_{i, j}$ thus defined are all distinct. By
  construction, $y_{i, j} \perp y_{j, i}$. Further, the $y_{i, j}$ exhaust $Y$
  because $\Gamma$ contains exactly $n (n-1)$ long vertices. Given two
  vertices $y_{i, j}$ and $y_{k, l}$, we find $m \in \{1, 2, \ldots, n\}
  \backslash\{i, j, k, l\}$ whence $y_{i, j}$ and $y_{k, l}$ are both
  contained in $x_m^{\perp} \cong \mathbbm{W}(B_{n - 1})$. $y_{i, j}$ is
  characterized in $x_m^{\perp}$ as one of the two long vertices contained in
  $\{x_r : r \nin \{i, j, m\}\}^{\perp}$. Likewise, $y_{k, l}$ is
  characterized in $x_m^{\perp}$ as one of the two long vertices contained in
  $\{x_r : r \nin \{k, l, m\}\}^{\perp}$. Consequently, for $\{i, j\} \neq
  \{k, l\}$, $y_{i, j} \perp y_{k, l}$ if and only if $\{i, j\} \cap \{k, l\}=
  \emptyset$. Hence, $\Gamma \cong \mathbbm{W}(B_n)$.
  
  Finally, consider the case $n = 5$. We still find that each short component
  is a clique on $5$ vertices. Let $X$ be one such short component. We count
  that there are $20$ long vertices neighbored to one of the vertices of $X$.
  On the other hand, we see again that each long component has short neighbors
  in only one short component. Accordingly, the $20$ long neighbors of $X$
  constitute a union of long components. However, a long component is locally
  $K_1  \sqcup  3 \cdot K_2$ and therefore has at least $12$ vertices. We
  conclude that there is only one long component $Y$ with vertices neighbored
  to $X$. Now, the remainder of the preceding argument applies and shows that
  $\Gamma \cong \mathbbm{W}(B_5)$ as claimed.
\end{proof}

The case $n = 4$ of Theorem \ref{thm-lr-recognizewbn} is discussed in Remark
\ref{rk-locallylikewb4} where it is shown that there are infinitely many
finite connected bichromatic graphs that are locally like $\mathbbm{W}(B_4)$.
The case of type $F_4$ is discussed in detail in the next section. Note that the Weyl graph $\mathbbm{W}(G_2)$ is isomorphic to three disjoint edges
of mixed type.

\section{\label{sct-wf4}Local recognition of $\mathbbm{W}(F_4)$}

\subsection{Graphs locally like $\mathbbm{W}(F_4)$}\label{4.1}

The Weyl graph $\mathbbm{W}(F_4)$ is a connected bichromatic locally homogeneous
graph on $24$ vertices with short local graph $\mathbbm{W}(B_3)$ and long
local graph $\mathbbm{W}(C_3)$. As we will see shortly, $\mathbbm{W}(F_4)$ is not locally recognizable.
Before we turn to investigating additional constraints under which we seek to
recognize $\mathbbm{W}(F_4)$ nonetheless, we study connected bichromatic
graphs $\Gamma$ which are locally like $\mathbbm{W}(F_4)$. The
results we obtain then guide our way in determining appropriate conditions
under which we are able to recognize $\mathbbm{W}(F_4)$ alongside its twisted
copy. An easy but crucial observation to start with is the following.

\begin{proposition}
  \label{thm-wf4-partitioninto4cliques}Let $\Gamma$ be locally like
  $\mathbbm{W}(F_4)$. The short (respectively long) induced subgraph of
  $\Gamma$ is isomorphic to a disjoint union of $4$-cliques.
\qed\end{proposition}

Let $\Gamma$ be a bichromatic graph that is locally like $\mathbbm{W}(F_4)$.
Observe that the graph obtained from $\Gamma$ by exchanging the roles of short
and long vertices is locally like $\mathbbm{W}(F_4)$ as well. Results that we
obtain for short vertices of graphs locally like $\mathbbm{W}(F_4)$ are
therefore also true for long vertices.

Paraphrasing Proposition \ref{thm-wf4-partitioninto4cliques}, the vertices of
$\Gamma$ come in $4$-cliques of the same type. In order to simplify things it
is natural to collapse these $4$-cliques into single vertices.

\begin{definition}
  Let $\Lambda$ be a graph and $\Pi$ a partition of its vertices. The
  {\em contraction} $\Lambda / \Pi$ is
  the graph on $\Pi$ such that two sets $A, B \in \Pi$ are adjacent whenever
  there is $a \in A$ and $b \in B$ which are adjacent in $\Lambda$. If
  $\Lambda$ is bichromatic then $\Pi$ is required to partition into sets of
  short and long vertices and $\Lambda / \Pi$ is a bichromatic graph in the
  natural way.
\end{definition}

In this language, we thus investigate the collapsed graph $\Gamma / \Pi$ where
$\Pi$ is the partition of $\Gamma$ into short and long $4$-cliques. To this
end, we analyze how these $4$-cliques relate to each other.

\begin{proposition}
  \label{thm-wf4-relationof4cliques}Let $\Gamma$ be locally like
  $\mathbbm{W}(F_4)$, and $x_1, x_2, x_3, x_4$ a short $4$-clique in $\Gamma$.
  Let $i \neq j$ and $k \neq l$.
  \begin{itemize}
    \item $\{x_i, x_j \}^{\perp}$ is locally $K_2^s  \sqcup  K_2^{\ell}$. In
    particular, for any pair $x_i, x_j$ there exist unique long vertices
    $y_{i, j}, y_{j, i}$ contained in $\{x_i, x_j \}^{\perp}$.
    
    \item $\{x_i, x_j, x_k \}^{\perp}$ contains no long vertex if $i, j, k$
    are distinct. In particular, the vertices $y_{i, j}$ are all distinct.
    
    \item There are exactly $12$ long vertices adjacent to at least one of the
    $x_i$, namely the above vertices $y_{i, j}$.
    
    \item $y_{i, j} \perp y_{k, l}$ implies that $\{k, l\}=\{i, j\}$ or $\{k,
    l\} \cap \{i, j\}= \emptyset$.
  \end{itemize}
\end{proposition}

\begin{proof}
  Exploiting the local structure at $x_i$ we see that every short adjacent
  pair $x_i, x_j$ has exactly two long neighbors in common which we will
  (arbitrarily) denote by $y_{i, j}$ and $y_{j, i}$. Accordingly, $y_{i, j}
  \perp y_{j, i}$. Looking at the neighbors of a vertex $y_{i, j}$ reveals
  that $x_i$ and $x_j$ are the only short vertices among $x_1, x_2, x_3, x_4$
  which are adjacent to $y_{i, j}$. Consequently, the $y_{i, j}$ are $12$
  distinct vertices. Since three adjacent short vertices share no long
  neighbors we count that exactly
  \[ \binom{4}{1} 6 - \binom{4}{2} 2 = 12 \]
  long vertices are neighbored to at least one of the vertices $x_1, x_2, x_3,
  x_4$. Consequently, the long neighbors of the $x_i$ are precisely the
  vertices $y_{i, j}$. For the last claim, assume that $y_{i, j} \perp y_{k,
  l}$ and $\{k, l\} \cap \{i, j\}=\{i_0 \}$. A look at the neighbors of
  $x_{i_0}$ shows that this is a contradiction.
\end{proof}

If $\Gamma$ is locally like $\mathbbm{W}(F_4)$ and $\Pi$ is the partition of
$\Gamma$ into short and long $4$-cliques, then we add the following extra
structure to the collapsed graph $\Gamma / \Pi$. Two vertices $X, Y \in \Gamma
/ \Pi$ are said to be {\em strongly connected} if every $x \in X$ is at
distance $1$ from $Y$ in $\Gamma$ and vice versa. In this case, we think of
$X$ and $Y$ as being connected by two edges, the reason of which will be clear
from the next proposition. The number of neighbors of $X$ where we count those
neighbors twice that are strongly connected to $X$ is said to be the
{\em bivalency} of $X$.

\begin{proposition}
  \label{thm-wf4-contractionlocally}Let $\Gamma$ be locally like
  $\mathbbm{W}(F_4)$, and let $\Pi$ be the partition of $\Gamma$ into short
  and long $4$-cliques. The contraction $\Gamma / \Pi$ is
  bipartite of bivalency $6$.
\end{proposition}

\begin{proof}
  Let $X \in \Gamma / \Pi$ be a short vertex. By Proposition
  \ref{thm-wf4-partitioninto4cliques}, $X$ has only long neighbors. $X =\{x_1,
  x_2, x_3, x_4 \}$ is a $4$-clique of $\Gamma$ and according to Proposition
  \ref{thm-wf4-relationof4cliques} there are $12$ long vertices $y_{i, j}$ at
  distance $1$ from $X$ in $\Gamma$. Each pair of vertices $y_{i, j}$, $y_{j,
  i}$ is contained in exactly one long neighbor $Y_{\{i, j\}}$ of $X$. Let $k,
  l$ be the indices such that $\{i, j, k, l\}=\{1, 2, 3, 4\}$. Then, by
  Proposition \ref{thm-wf4-relationof4cliques}, either $Y_{\{k, l\}} \neq
  Y_{\{i, j\}}$, in which case both long vertices are connected to $X$ by just
  one edge, or $Y_{\{k, l\}} = Y_{\{i, j\}}$, in which case both long vertices
  are connected to $X$ by two edges. In any case, we count that the bivalency
  of $X$ is $6$.
\end{proof}

We now do the reverse and prove that every bipartite graph of bivalency $6$ is
the contraction of some graph which is locally like $\mathbbm{W}(F_4)$. Note,
however, that non-isomorphic graphs locally like $\mathbbm{W}(F_4)$ can have
isomorphic contractions.

\begin{lemma}
  \label{thm-wf4-fromlocN6}For every connected bipartite graph $\Lambda$ of
  bivalency $6$ there is a connected bichromatic graph $\Gamma$ that is
  locally like $\mathbbm{W}(F_4)$ such that $\Gamma / \Pi = \Lambda$ where
  $\Pi$ is the partition of $\Gamma$ into short and long
  $4$-cliques.
\end{lemma}

\begin{proof}
  Let $\Lambda$ be a bipartite graph of bivalency $6$. Exploiting that
  $\Lambda$ is $2$-colorable, we may identify $\Lambda$ with a bichromatic
  graph such that no two vertices of the same type are adjacent. For any
  vertex $x$ of $\Lambda$ choose an injection
  \[ x^{\perp} \rightarrow \binom{4}{2}, \hspace{2em} y \mapsto a (x, y) \]
  from its neighbors to the six $2$-subsets of $\{1, 2, 3, 4\}$ such that for
  strongly connected vertices $x, y$ the complement of $a (x, y)$ is not
  attained. This is always possible since $\Lambda$ has bivalency $6$. To
  every directed edge $(x, y)$ we thus assigned the $2$-subset $a (x, y)$ of
  $\{1, 2, 3, 4\}$. Construct the bichromatic graph $\Gamma$ from $\Lambda$ as
  follows. For every vertex $x \in \Lambda$ add a $4$-clique $x_1, x_2, x_3,
  x_4$ of the same type as $x$ to $\Gamma$. Then, for $x, y \in \Lambda$ let
  $x_i$ and $y_j$ be adjacent in $\Gamma$ if and only if $x$ and $y$ are
  adjacent in $\Lambda$ and the following holds: either $(i, j) \in a (x, y)
  \times a (y, x)$, or $x$ and $y$ are strongly connected and $(i, j) \in
  \overline{a (x, y)} \times \overline{a (y, x)}$. By construction,
  contracting the $4$-cliques of $\Gamma$ produces $\Lambda$. It is
  straightforward to check that $\Gamma$ is locally like $\mathbbm{W}(F_4)$.
\end{proof}

\begin{corollary}
  \label{thm-infinitelymanygraphslocallylikewf4}There exist infinitely many
  non-isomorphic finite connected bichromatic graphs that are locally like
  $\mathbbm{W}(F_4)$.
\end{corollary}

\begin{proof}
  We claim that there are infinitely many finite connected bipartite graphs
  $\Lambda$ that are locally $\overline{K_6}$ and hence of bivalency $6$ (if
  no vertices are assumed to be strongly connected). To this end, note that
  the graphs $C_k \times C_m \times C_n$ are connected and locally
  $\overline{K_6}$ for $k, m, n \geqslant 4$. Since cycles $C_n$ are $2$-colorable
  whenever $n$ is even, the graphs $C_k \times C_m \times C_n$ are
  $2$-colorable and hence bipartite whenever $k$, $m$, $n$ are all even. The
  claim follows from Lemma \ref{thm-wf4-fromlocN6}.
\end{proof}

\begin{remark}
  \label{rk-locallylikewb4}Analogous to Lemma \ref{thm-wf4-fromlocN6} one
  shows that for every connected bipartite graph $\Lambda$ of bivalency $(2,
  6)$ (meaning that vertices of one type have valency $2$ and vertices of the
  other type have valency $6$) there is a connected bichromatic graph $\Gamma$
  that is locally like $\mathbbm{W}(B_4)$ such that $\Gamma / \Pi = \Lambda$
  where $\Pi$ is again the partition of $\Gamma$ into short and long
  $4$-cliques. This easily implies that there are infinitely many finite
  connected bichromatic graphs that are locally like $\mathbbm{W}(B_4)$.
\end{remark}

Let $\Gamma$ be locally like $\mathbbm{W}(F_4)$ and assume that the collapsed
graph $\Gamma / \Pi$ contains strongly connected vertices $X$ and $Y$. This means
that, say, $X = {x_1,x_2,x_3,x_4}$ are short vertices, $Y = {y_1,y_2,y_3,y_4}$
are long vertices, and we have the adjacencies
\[ x_1,x_2 \perp y_1,y_2, \hspace{4em} x_3,x_4 \perp y_3,y_4. \]
It is straightforward to check that replacing these by
\[ x_1,x_2 \perp y_3,y_4, \hspace{4em} x_3,x_4 \perp y_1,y_2 \]
produces a graph $\Gamma^\prime$ which is also locally like $\mathbbm{W}(F_4)$.
We say that $\Gamma^\prime$ is a {\em twisted copy} of $\Gamma$. In particular, for
$\Gamma = \mathbbm{W}(F_4)$ these twisted copies are all isomorphic and we denote
the resulting graph by $\mathbbm{W}(F_4)^\prime$.

\subsection{Recognition results}

We now discuss further properties of the Weyl graph $\mathbbm{W}(F_4)$ in order
to characterize $\mathbbm{W}(F_4)$ among the connected bichromatic graphs that
are locally like $\mathbbm{W}(F_4)$. For more details we refer to the thesis
{\cite{straub-diploma}} of the third author. We start with some easy
observations.

\begin{proposition}
  \label{thm-wf4-nrofvertices}Let $\Gamma$ be a finite bichromatic graph
  that is locally like $\mathbbm{W}(F_4)$. Then the numbers of short and long
  vertices in $\Gamma$ are the same, $| \Gamma |$ is divisible by $8$ and
  $| \Gamma | \geqslant 24$.
\qed\end{proposition}

Since $|\mathbbm{W}(F_4) | = 24$ we see that, in a sense, $\mathbbm{W}(F_4)$
is maximally tight among the graphs that are locally like $\mathbbm{W}(F_4)$.
There are several further properties of a graph, for instance its diameter,
that describe its tightness. The following notion of tight connectedness is
another way to express tightness of a bichromatic graph.

\begin{definition}
  A bichromatic graph is said to be {\em tightly connected} if every long
  vertex has a neighbor in every short component and vice versa.
\end{definition}

These three notions of tightness, however, are not local in nature (where a
local property is meant to be one which can be expressed in terms of the
neighbors of each vertex). In order to find a more local notion to
describe the tightness of $\mathbbm{W}(F_4)$ we investigate the relation of
vertices at distance $2$. The following is straightforward to check.

\begin{proposition}
  \label{thm-wf4-verticesatdistance2}Let $\Gamma$ be locally like
  $\mathbbm{W}(F_4)$, and let $x, y \in \Gamma$ be at distance $2$.
  \begin{itemize}
    \item If $x, y$ are both short (respectively long) vertices then $\{x,
    y\}^{\perp} \cong \mu (x, y) \cdot K^{\ell}_1$ (respectively $K_1^s$) for
    some $\mu (x, y) \in \{1, 2, 3\}$.
    
    \item If $x, y$ are of mixed type then $\{x, y\}^{\perp} \cong \mu_s (x,
    y) \cdot K_2^s  \sqcup  \mu_{\ell} (x, y) \cdot K_2^{\ell}$ for some
    $\mu_s (x, y) + \mu_{\ell} (x, y) \in \{1, 2\}$.
  \end{itemize}
\end{proposition}

For the Weyl graph $\mathbbm{W}(F_4)$ the parameters $\mu, \mu_s,
\mu_{\ell}$ defined in Proposition
\ref{thm-wf4-verticesatdistance2} are constant and take the maximum possible
values $\mu = 3$ and $\mu_s = \mu_{\ell} = 1$ which is another, more local,
instantiation of the tightness of $\mathbbm{W}(F_4)$. Notice that the condition
$\mu_s = \mu_{\ell} = 1$ is equivalent to the contraction $\Gamma / \Pi$,
studied in Proposition \ref{thm-wf4-contractionlocally}, being
locally homogeneous with $\Delta_s (\Gamma / \Pi) \cong \overline{K_3}^{\ell}$
and $\Delta_{\ell} (\Gamma / \Pi) \cong \overline{K_3}^s$.

The following theorem summarizes our recognition results for the Weyl graph
$\mathbbm{W}(F_4)$. Note that all of the provided conditions under which a
graph $\Gamma$ is almost recognized as $\mathbbm{W}(F_4)$ are statements which
describe the tightness of $\Gamma$.

\begin{theorem}
  \label{thm-recognizewf4}Let $\Gamma$ be a connected bichromatic graph that
  is locally like $\mathbbm{W}(F_4)$. Assume that
  \begin{itemize}
    \item $| \Gamma | = 24$, or
    
    \item $\Gamma$ is tightly connected, or
    
    \item $\Gamma$ has diameter $2$, or
    
    \item $\mu = 3$.
  \end{itemize}
  If one of these conditions holds then $\Gamma$ is isomorphic to
  $\mathbbm{W}(F_4)$ or to its twisted copy $\mathbbm{W}(F_4)^\prime$. In particular, $\tmop{Aut}
  (\Gamma) \cong W (F_4) / Z$ where $Z$ denotes the center of $W (F_4)$.
\end{theorem}

We prove Theorem \ref{thm-recognizewf4} by a series of propositions. The proof of
the case $\mu = 3$ is similar in spirit to the previous ones. It is therefore omitted; the interested reader is referred to {\cite{straub-diploma}} for the details.

\begin{proposition}
  \label{thm-wf4order24}Let $\Gamma$ be a connected bichromatic graph that is
  locally like $\mathbbm{W}(F_4)$. If $| \Gamma | = 24$ then $\Gamma \cong
  \mathbbm{W}(F_4)$ or $\Gamma \cong \mathbbm{W}(F_4)^\prime$. Further, $\tmop{Aut}
  (\Gamma) \cong W (F_4) / Z$.
\end{proposition}

\begin{proof}
  As observed in Proposition \ref{thm-wf4-nrofvertices}, every graph that
  is locally like $\mathbbm{W}(F_4)$ has at least $12$ short and $12$ long
  vertices. $\Gamma$ therefore consists of exactly $12$ vertices of each type.
  
  Let $x_1, x_2, x_3, x_4$ be a short $4$-clique. Adopting the notation of
  Proposition \ref{thm-wf4-relationof4cliques}, let $y_{i, j}$ and $y_{j, i}$
  be the long vertices adjacent to both $x_i$ and $x_j$. The $y_{i, j}$ are
  $12$ distinct vertices and therefore constitute the long vertices of
  $\Gamma$. It follows from Proposition \ref{thm-wf4-relationof4cliques} that
  the three long $4$-cliques are given by $y_{i, j}, y_{j, i}, y_{k, l}, y_{l,
  k}$ for disjoint $\{i, j\}$ and $\{k, l\}$.
  
  Each of the remaining eight short vertices has exactly two long neighbors in
  each of the three long $4$-cliques. Let $x_5$ be one of remaining short
  vertices. The two neighbors of $x_5$ in a $4$-clique $y_{i, j}, y_{j, i},
  y_{k, l}, y_{l, k}$ are one of $y_{i, j}, y_{j, i}$ along with one of $y_{k,
  l}, y_{l, k}$. We ambiguously defined the vertices $y_{i, j}, y_{j, i}$ as
  the long vertices contained in $\{x_i, x_j \}^{\perp}$ so we may as well
  assume that $x_5$ is adjacent to $y_{i, j}$ and $y_{k, l}$ with $i < j$ and
  $k < l$. Let $x_6$ be the unique short vertex also adjacent to $y_{1, 2},
  y_{3, 4}$. Likewise, let $x_7$ be the short vertex also adjacent to $y_{1,
  3}, y_{2, 4}$, and $x_8$ the short vertex also adjacent to $y_{1, 4}, y_{2,
  3}$. By construction, $x_5, x_6, x_7, x_8$ is a $4$-clique. Notice that for
  instance $x_5, x_6 \in \{y_{1, 2}, y_{3, 4} \}^{\perp}$ implies that $x_7,
  x_8 \in \{y_{2, 1}, y_{4, 3} \}^{\perp}$. Altogether this determines the
  induced subgraph on $x_1, x_2, \ldots, x_8$ along with the vertices $y_{i,
  j}$.
  
  Let $x_9, x_{10}, x_{11}, x_{12}$ be the remaining short $4$-clique. We may
  assume that $x_9, x_{10}$ are the short vertices contained in $\{y_{1, 2},
  y_{4, 3} \}^{\perp}$. Accordingly, $x_{11}, x_{12} \in \{y_{2, 1}, y_{3, 4}
  \}^{\perp}$. We may also assume that $x_9$ is contained in $\{y_{1, 3},
  y_{4, 2} \}^{\perp}$ (because if both $x_9$ and $x_{10}$ were not contained
  in $\{y_{1, 3}, y_{4, 2} \}^{\perp}$ then both $x_{11}, x_{12} \in \{y_{1,
  3}, y_{4, 2} \}^{\perp}$ which contradicts $x_{11}, x_{12} \in \{y_{2, 1},
  y_{3, 4} \}^{\perp}$). Further, we may assume that $x_{11}$ is the second
  short vertex contained in $\{y_{1, 3}, y_{4, 2} \}^{\perp}$. Consider the
  two short vertices in $\{y_{1, 4}, y_{3, 2} \}^{\perp}$. These can be either
  $x_9, x_{12}$ or $x_{10}, x_{11}$, and either choice determines $\Gamma$.
  Denote with $\Gamma_1$ the graph corresponding to the choice $x_9,
  x_{12} \in \{y_{1, 4}, y_{3, 2} \}^{\perp}$, and with $\Gamma_2$ the
  graph corresponding to the choice $x_{10}, x_{11} \in \{y_{1, 4}, y_{3, 2}
  \}^{\perp}$. The following table summarizes adjacency involving the vertices
  $x_9, x_{10}, x_{11}, x_{12}$.
  
  \begin{center}
    \begin{tabular}{lll}
      by construction & $x_9, x_{10} \perp y_{1, 2}, y_{4, 3}$ & $x_{11},
      x_{12} \perp y_{2, 1}, y_{3, 4}$\\
      & $x_9, x_{11} \perp y_{1, 3}, y_{4, 2}$ & $x_{10}, x_{12} \perp y_{3,
      1}, y_{2, 4}$\\
      $\Gamma_1$ & $x_9, x_{12} \perp y_{1, 4}, y_{3, 2}$ & $x_{10},
      x_{11} \perp y_{4, 1}, y_{2, 3}$\\
      $\Gamma_2$ & $x_9, x_{12} \perp y_{4, 1}, y_{2, 3}$ & $x_{10},
      x_{11} \perp y_{1, 4}, y_{3, 2}$
    \end{tabular}
  \end{center}
  
  An implementation in the computer algebra system SAGE, see {\cite{sage28}},
  of the graphs $\Gamma_1$ and $\Gamma_2$ can be found in the
  appendix of {\cite{straub-diploma}}. In particular, it is verified that
  $\Gamma_1$ and $\Gamma_2$ are non-isomorphic, that the automorphism
  group of both graphs is isomorphic to $W (F_4) / Z$, and that $\Gamma_1$
  is isomorphic to $\mathbbm{W}(F_4)$. Accordingly, $\Gamma_2$
  is isomorphic to the twisted copy $\mathbbm{W}(F_4)^\prime$.
\end{proof}

\begin{proposition}
  \label{thm-wf4tightlyconnected}Let $\Gamma$ be a connected bichromatic graph
  that is locally like $\mathbbm{W}(F_4)$. If $\Gamma$ is tightly
  connected then $\Gamma \cong \mathbbm{W}(F_4)$ or
  $\Gamma \cong \mathbbm{W}(F_4)^\prime$.
\end{proposition}

\begin{proof}
  Fix a short $4$-clique $x_1, x_2, x_3, x_4$. Because of tightness every long
  vertex is adjacent to one of the $x_i$, and by Proposition
  \ref{thm-wf4-relationof4cliques} there are exactly $12$ such long vertices.
  Thus $\Gamma$ consists of $12$ long vertices. Likewise, $\Gamma$ contains
  exactly $12$ short vertices. Hence, $| \Gamma | = 24$, and the claim follows
  from Proposition \ref{thm-wf4order24}.
\end{proof}

\begin{proposition}
  Let $\Gamma$ be a connected bichromatic graph that is locally like
  $\mathbbm{W}(F_4)$. If $\Gamma$ has diameter $2$ then $\Gamma \cong
  \mathbbm{W}(F_4)$ or $\Gamma \cong \mathbbm{W}(F_4)^\prime$.
\end{proposition}

\begin{proof}
  Let $x_1, x_2, x_3, x_4$ be a short $4$-clique. As in Proposition
  \ref{thm-wf4-relationof4cliques} let $y_{i, j}, y_{j, i}$ be the long
  vertices adjacent to both $x_i$ and $x_j$. Assume that there is a long
  vertex $v$ which is not among the $12$ long vertices $y_{i, j}$. Because $v$
  is not adjacent to any of the $x_i$ and since the diameter of $\Gamma$ is
  $2$, we find a long vertex that connects $x_1$ and $v$. Without loss of
  generality let this long vertex be $y_{1, 2}$. This prevents $y_{1, 2},
  y_{2, 1}, y_{3, 4}, y_{4, 3}$ from forming a long $4$-clique. By Proposition
  \ref{thm-wf4-relationof4cliques} there are thus long vertices $v_1, v_2$ not
  among the $y_{i, j}$ such that $y_{3, 4}, y_{4, 3}, v_1, v_2$ form a long
  $4$-clique. Again, $v_1$ is not adjacent to any of the $x_i$ and hence is
  connected to $x_1$ by a long vertex. This is a contradiction since the long
  vertices adjacent to $x_1$ are the vertices $y_{1, j}$, $y_{j, 1}$.
  
  Consequently, $\Gamma$ contains no further long vertices besides the $12$
  vertices $y_{i, j}$, and hence, by Proposition \ref{thm-wf4-nrofvertices},
  $| \Gamma | = 24$. Apply Proposition \ref{thm-wf4order24}.
\end{proof}

\section{\label{sct-applications}group-theoretic applications}

Our guiding example for the application of the local recognition of graphs in group theory is the
characterization of the symmetric groups by means of the structure of its
transposition centralizers; cf.\ {\cite[Theorem 27.1]{gorenstein-classification1}}. A detailed proof of \cite[Theorem 27.1]{gorenstein-classification1}
 using local recognition results for the Weyl graphs of type $A_n$ is contained in the third author's thesis
{\cite{straub-diploma}}; that proof runs along the lines of the proof of \cite[Theorem 1.2]{Cohen/Cuypers/Gramlich:2005}.
 An early example of such results
can be found in {\cite{mullineux-gt78}} which along with
{\cite{buekenhout-gt77}} has been one of the original motivations for the second
author to pursue the local recognition of Kneser graphs
in {\cite{hall-gt87}}. Another fundamental example for this theme is \cite{Pasechnik:1994}.
%
%\begin{theorem}
%  \label{thm-ga-recognizewan}Let $n \geqslant 6$, and let $G$ be a group with
%  involutions $x, y$ such that
%  \begin{itemize}
%    \item $C_G (x) = \langle x \rangle \times J$ with $J \cong W (A_n)$,
%    
%    \item $C_G (y) = \langle y \rangle \times K$ with $K \cong W (A_n)$,
%    
%    \item $x$ is a reflection in $K$,
%    
%    \item $y$ is a reflection in $J$,
%    
%    \item $J \cap K$ contains an involution $z$ that is a reflection in both
%    $J$ and $K$.
%  \end{itemize}
%  If $G = \langle J, K \rangle$, then $G \cong W (A_{n + 2})$.
%\end{theorem}

Likewise, the
recognition results for the Weyl graph of type $F_4$ discussed in the previous section
give rise to the following local characterization of the Weyl group $W(F_4)$. Again,
we refer to {\cite{straub-diploma}} for a proof inspired by \cite{Cohen/Cuypers/Gramlich:2005}.

\begin{theorem}
  \label{thm-ga-recognizewf4}Let $G$ be a group with non-conjugate involutions
  $x, y$ such that
  \begin{itemize}
    \item $C_G (x) = \langle x \rangle \times J$ with $J \cong W (B_3)$,
    
    \item $C_G (y) = \langle y \rangle \times K$ with $K \cong W (C_3)$,
    
    \item $x$ (respectively $y$) is a short (respectively long) root
    reflection in $K$ (respectively $J$),
    
    \item $J \cap K$ contains involutions $x_1, y_1$ such that $x_1$
    (respectively $y_1$) is a short (respectively long) root reflection in $K$
    as well as in $J$, and
    
    \item there are a long root reflection $y_0 \neq y, y_1$ in $J$ and a
    short root reflection $x_0 \neq x, x_1$ in $K$ such that $x_0$ and $y_0$
    commute.
  \end{itemize}
  If $G = \langle J, K \rangle$ then $G \cong W (F_4)$.
\end{theorem}

The interest in group-theoretic local recognition
results like Theorem \ref{thm-ga-recognizewf4} stems from the classification of finite simple groups (outlined in
{\cite{gorenstein-classification1}}) and the fact that the majority of finite
simple groups arises from (twisted) Chevalley groups. These can be
defined as groups generated
by subgroups isomorphic to $\tmop{SL} (2, q)$ subject to certain relations by the Curtis--Tits theorem formulated as in {\cite{phan-gt70}}, {\cite{humphreys-gt72}}, {\cite{timmesfeld-gt04}}, by Phan's theorems \cite{Phan:1977}, \cite{Phan:1977b}, and by the Phan-type theorems \cite{Bennett/Gramlich/Hoffman/Shpectorov:2007}, \cite{Gramlich/Hoffman/Shpectorov:2003}, \cite{Gramlich/Witzel}. 
Recently, see
{\cite[Local recognition theorem~2]{gramlich-phan}}, Kristina Altmann and the first author proved a local
recognition result for Chevalley groups of (twisted) type $A_7$ and $E_6$ based on
results and techniques of {\cite{altmann-phd}}; this result makes serious use of the local
recognition of graphs that are locally the Weyl graph of type $A_5$ and of the Curtis--Tits theorem and Phan's theorems. We hope that our
analysis can help to approach a similar recognition result for Chevalley
groups of type $F_4$ based on the Phan-type theorem of type $F_4$ proved by Hoffman, M\"uhlherr, Shpectorov and the first author and published in \cite{Gramlich/Witzel}. For more details we refer to the thesis
{\cite{straub-diploma}} of the third author and the survey \cite{gramlich-phan} of the first author.

\end{document}